\documentclass[11pt,reqno,a4paper]{amsart}

\bibliography{mybibliography}

\usepackage{mathabx}
\usepackage{enumerate}
\usepackage{amssymb}
\usepackage{amsmath}
\usepackage{amscd}
\usepackage{amsthm}
\usepackage{amsfonts}
\usepackage{graphicx}
\usepackage[all]{xy}
\usepackage{verbatim}
\usepackage{hyperref}
\usepackage{amssymb}

\usepackage[T1]{fontenc}

\newtheorem{theorem}{Theorem}[section]

\newtheorem*{plunnecke}{Pl\"unnecke's Theorem}

\newtheorem*{correspondence}{Correspondence Principle I}
\newtheorem*{corr2}{Correspondence Principle II}
\newtheorem*{corr22}{Correspondence Principle}

\newtheorem{lemma}{Lemma}[section]
\newtheorem{corollary}[lemma]{Corollary}
\newtheorem{proposition}{Proposition}[section]

\theoremstyle{definition}
\newtheorem{definition}{Definition}[section]

\theoremstyle{definition}

\theoremstyle{definition}

\theoremstyle{definition}
\theoremstyle{definition}\newtheorem{remark}[theorem]{Remark}
\theoremstyle{definition}

\newcommand{\eps}{\epsilon}

\newcommand{\cC}{\mathcal{C}}

\newcommand{\cE}{\mathcal{E}}

\newcommand{\cL}{\mathcal{L}}
\newcommand{\cM}{\mathcal{M}}

\newcommand{\cP}{\mathcal{P}}

\newcommand{\cS}{\mathcal{S}}

\newcommand{\bZ}{\mathbb{Z}}

\newcommand{\bN}{\mathbb{N}}

\newcommand{\ra}{\rightarrow}

\newcommand{\onto}{\xymatrix{\ar@{>>}[r]&}}
\newcommand{\da}[4]{\xymatrix{#1 \ar@<.5ex>[r]^{#2} \ar@<-.5ex>[r]_{#3} & #4}}

\newcommand{\qand}{\quad \textrm{and} \quad}
\newcommand{\qqand}{\qquad \textrm{and} \qquad}

\numberwithin{equation}{section}

\begin{document}

\title{Pl\"unnecke Inequalities for countable abelian groups}


\author{Michael Bj\"orklund}
\address{Department of Mathematics , ETH Z\"urich, Z\"urich, Switzerland}
\curraddr{Department of Mathematics, Chalmers, Gothenburg, Sweden}
\email{micbjo@chalmers.se}
\thanks{}

\author{Alexander Fish}
\address{School of Mathematics and Statistics, University of Sydney, Australia}
\curraddr{}
\email{alexander.fish@sydney.edu.au}
\thanks{}

\subjclass[2010]{Primary: 37B05; Secondary: 11B13, 11K70 }

\keywords{Ergodic ramsey theory, additive combinatorics}

\date{}

\dedicatory{}

\begin{abstract} 
We establish in this paper a new form of Pl\"unnecke-type inequalities for ergodic probability measure-preserving actions of any countable abelian group. Using a correspondence principle for product sets, this allows us to deduce lower bounds on the upper and lower Banach densities of  
any product set in terms of the upper Banach density of an iterated 
product set of one of its addends. These bounds are new already in the
case of the integers.

We also introduce the notion of an ergodic basis, which is parallel, but significantly weaker than the analogous notion of an additive basis, and deduce Pl\"unnecke bounds on their impact functions with respect to 
both the upper and lower Banach densities on any countable abelian 
group. 
\end{abstract}

\maketitle


\section{Introduction}

\subsection{General comments}
\label{general}
Let $G$ be a countable group and suppose $A, B \subset G$ are 
non-empty subsets. We define the \emph{product set} $AB$ by
\[
AB = \Big\{ ab \, : \, a \in A, \: \:  b \in B \Big\} \subset G.
\]
Let $\cM(G)$ denote the set of means on $G$, i.e. the convex set of all 
positive norm-one functionals on the C*-algebra $\ell^{\infty}(G)$. 
Note that every $\lambda$ in $\cM(G)$ gives rise to a
\emph{finitely additive} probability measure $\lambda'$ 
on the group $G$ via the formula
\begin{equation}
\label{mean}
\lambda'(B) = \lambda(\chi_B), \quad B \subset G,
\end{equation}
where $\chi_B$ denotes the indicator function on the set $B$. 
Given a set $\cC \subset \cM(G)$, we define the \emph{upper}
and \emph{lower Banach densities} of a set $B \subset G$ 
with respect to $\cC$ by
\[
d^{*}_{\cC}(B) = \sup_{\lambda \in \cC} \lambda'(B)
\qqand
d_*^{\cC}(B) = \inf_{\lambda \in \cC} \lambda'(B),
\]
respectively. Fix $A \subset G$ and $\cC \subset \cM(G)$ and define the 
\emph{upper} and \emph{lower impact functions} with respect 
to $A$ and $\cC$ by
\[
\cC_A^{*}(t) = 
\inf\Big\{ d^{*}_{\cC}(AB) \, : \, d^{*}_{\cC}(B) \geq t \Big\}
\qand
\cC^{A}_*(t) = 
\inf\Big\{ d_{*}^{\cC}(AB) \, : \, d_{*}^{\cC}(B) \geq t \Big\},
\]
for $0 \leq t \leq 1$, respectively. A fundamental problem in 
additive combinatorics is to understand the behavior of these 
functions for various classes of sets $A \subset G$ and 
$\cC \subset \cM(G)$. \\

In the case of the additive group $\bZ$ of integers, a classically important subset of $\cM(\bZ)$ is the set $\cS$ of \emph{Birkhoff means}. We say that $\lambda \in \cM(\bZ)$ is a \emph{Birkhoff mean} if it is a 
weak*-cluster point of the sequence $(\lambda_n)$ of means
on $\bZ$ defined by
\[
\lambda_n(\varphi) = \frac{1}{n} \sum_{k=0}^{n-1} \varphi(k),
\quad \varphi \in \ell^{\infty}(\bZ).
\]
One readily checks that every Birkhoff mean is \emph{invariant}, i.e. $\lambda'(gB) = \lambda'(B)$ for all $B \subset \bZ$ and $g$ in $\bZ$. \\

We should warn the reader that the associated upper and lower Banach 
densities with respect to $\cS$ are often referred to as the \emph{upper} and \emph{lower asymptotic densities} respectively in the literature, and are usually defined by
\[
d^{*}_{\cS}(B) = \varlimsup_{n \ra \infty} \frac{|B \cap [0,n]|}{n+1} 
\qqand 
d_*^{\cS}(B) = \varliminf_{n \ra \infty} \frac{|B \cap [0,n]|}{n+1}
\]
respectively. \\

Given a set $A \subset G$, we denote by $A^{k}$ the $k$-fold product
set of $A$ with itself. We shall say that $A$ is a \emph{basis of order $k$} with respect to $\cC$ if $d^{*}_{\cC}(A^{k}) = 1$ and 
$A$ is a \emph{uniform basis of order $k$} with respect to $\cC$ if $d_*^{\cC}(A^{k}) = 1$. Clearly, every uniform basis of order $k$ with respect to $\cC$ is a basis of order $k$ with respect to $\cC$, but the converse does not hold in general. In the case when $G = \bZ$ and $\cC = \cS$, the terms \emph{upper} and \emph{lower asymptotic basis} are more commonly used in the literature.  

The following celebrated result by 
Pl\"unnecke (Satz 1.2 in \cite{Plunnecke70}) gives a non-trivial lower 
bound on $\cS_{*}^A$ when $A$ is a uniform basis of order $k$ (see
the proof of Theorem 7.2 in \cite{RuSS} for the easy derivation of this statement from Pl\"unnecke's original argument).

\begin{plunnecke}
Let $\cS \subset \cM(\bZ)$ denote the set of Birkhoff means on $\bZ$
and suppose $A \subset \bZ$ is a uniform basis of order $k$ with respect
to $\cS$. Then
\[
\cS^{A}_{*}(t) \geq t^{1-\frac{1}{k}},
\]
for all $0 \leq t \leq 1$.
\end{plunnecke}
On the other hand, Jin constructed in \cite{Jin11} a basis 
$A \subset \bZ$ of order $2$ with respect to $\cS$ such that 
\[
\cS^{*}_A\left(\frac{1}{2}\right) = \frac{1}{2}.
\]

\subsection{A correspondence principle for product sets}
\label{corr}
The main aim of this paper is to establish Pl\"unnecke bounds on
the impact functions with respect to the upper and lower Banach 
densities associated to the set $\cL_G$ of \emph{all} invariant means on any countable abelian group $G$. We shall introduce below the notion of
an \emph{ergodic basis of order $k$}, which is significantly weaker than
the notion of a basis which we discussed above. However, before we can
do this, we need to give the basic set up. \\

Let $G$ be a countable abelian group and let $\cL_G$ denote the set of 
all invariant means on $G$. By a classical theorem of Kakutani-Markov, 
this set is always non-empty. To avoid cluttering with sub-indices, we 
shall adopt the conventions
\[
d^{*} = d^*_{\cL_G} \qand d_* = d_*^{\cL_G}
\]
from now on, and simply refer to $d^{*}$ and $d_*$ as the \emph{upper} and \emph{lower Banach densities} on $G$ respectively.  

Let $(X,\mu)$ be a probability measure space such that the Hilbert space $L^2(X,\mu)$ is separable. If $G$ acts on $X$ by bi-measurable bijections
which preserve the measure $\mu$, then we say that $(X,\mu)$ is a 
\emph{$G$-space}. If $X$ in addition is compact and the $G$-action is 
by homeomorphisms which preserve $\mu$, then we say that it is
a \emph{compact $G$-space}. If there are no $G$-invariant 
measurable sets $B \subset X$ with $0 < \mu(B) < 1$, then we say 
that $\mu$ is an ergodic probability measure. If $A \subset G$
and $B \subset X$ is a measurable set, then we denote by $AB$ the union
of all the sets of form $aB$, where $a$ ranges over $A$. In particular,
if $\mu$ is an ergodic probability measure, then $\mu(GB)$ equals 
either zero or one, depending on whether $B$ is a $\mu$-null set or 
not.  

An important relation between these concepts and the Banach densities discussed earlier can be summarized in the following proposition which will be proved in the Appendix. 

\begin{corr22}
Let $G$ be a countable abelian group and suppose $A, B' \subset G$. 
Then there exists a compact metrizable space $X$, equipped with an 
action of $G$ by homeomorphisms, a clopen set $B \subset X$ and 
\emph{ergodic} $G$-invariant probability measures $\mu$ and $\nu$ 
on $X$ such that 
\[
d^{*}(B') = \mu(B) 
\qand 
d_*(B') \leq \nu(B)
\]
and
\[
d^{*}(AB') \geq \mu(AB) 
\qand 
d_*(AB') \geq \nu(AB).
\]
\end{corr22}

The following notion will play an important role in this paper. 

\begin{definition}[Ergodic set]
Let $(X,\mu)$ be an ergodic $G$-space. We say that a set $A \subset G$
is an \emph{ergodic set with respect to $(X,\mu)$} if $\mu(AB)$ equals 
one for every measurable set $B \subset X$ of positive $\mu$-measure. 
If $A \subset G$ is an ergodic set with respect to every ergodic $G$-space,
we simply say that $A$ is an \emph{ergodic set}.
\end{definition}

By definition, $G$ itself is an ergodic set. However, no proper subgroup 
of $G$ can be an ergodic set for all ergodic $G$-spaces. There are several
criteria which ensure that a set is ergodic. One of the most well-known involves the notion of \emph{equidistributed sets}. Recall that a set 
$A \subset G$ is \emph{equidistributed} if there exists an exhaustion $(A_n)$ of $A$ by finite sets such that 
\[
\lim_{n \ra \infty} \frac{1}{|A_n|} \sum_{g \in A_n} \chi(g) = 0
\]
for all non-trivial characters $\chi$ on $G$. Every equidistributed set
is ergodic, but does not need to have positive upper Banach density with
respect to $\cL_G$. For instance, the set 
\[
A = \Big\{ [n^{3/2}] \, : \, n \geq 1 \Big\} \subset \bZ,
\] 
where $[\cdot]$ denotes the integer part, is known to be 
equidistributed (see e.g. Theorem 1.3 in \cite{Bos}).

\subsection{Statements of the main results}
In order to state our results, we need the following definition. 

\begin{definition}[Ergodic basis]
Let $(X,\mu)$ be an ergodic $G$-space. We say that $A \subset G$
is an \emph{ergodic basis or order $k$ with respect to $(X,\mu)$} if $A^k$ is an 
ergodic set with respect to $(X,\mu)$. If $A$ is an ergodic basis or order 
$k$ with respect to every $G$-space, then we simply say that $A$ is an 
\emph{ergodic basis of order $k$}.
\end{definition}

Clearly, every basis of order $k$ with respect to $\cL_G$ (see Subsection \ref{general} for the definition) is an ergodic basis of order $k$. 

\begin{theorem}
\label{thm1}
Let $G$ be a countable abelian group and suppose $(X,\mu)$ is an ergodic
$G$-space. If $A \subset G$ is an ergodic basis of order $k$ with respect 
to $(X,\mu)$, then
\[
\mu(AB) \geq \mu(B)^{1-\frac{1}{k}}
\]
for every measurable subset $B \subset X$.
\end{theorem}

An application of the Correspondence Principle mentioned in Subsection \ref{corr} yields the following Pl\"unnecke bounds with respect to the upper and lower Banach densities on any countable abelian group. 

\begin{corollary}
\label{cor1}
Let $G$ be a countable abelian group and suppose $A \subset G$ is an
ergodic basis of order $k$. Then, 
\[
d^{*}(AB) \geq d^*(B)^{1-\frac{1}{k}} 
\qqand
d_*(AB) \geq d_*(B)^{1-\frac{1}{k}}
\]
for all $B \subset G$. In particular, this holds whenever the set $A^{k}$ is equidistributed in $G$.
\end{corollary}

Our methods also give the following general bound when no constraint 
is forced upon the set $A \subset G$.
\begin{theorem}
\label{thm2}
Let $G$ be a countable abelian group and suppose $(X,\mu)$ is
an ergodic $G$-space. Let $A \subset G$ and
suppose $B \subset X$ is a measurable subset such that 
\[
\mu(AB) \leq K \cdot \mu(B)
\]
for some $K$. Then, 
\[
d^*(A^k) \leq K^k \cdot \mu(B),
\]
for every integer $k$. In particular,
\[
\mu(AB) \geq d^*(A^k)^{\frac{1}{k}} \cdot \mu(B)^{1-\frac{1}{k}}
\]
for all $k$.
\end{theorem}

An application of the Correspondence Principle above now immediately yields the following two corollaries of Theorem \ref{thm2}.

\begin{corollary}
\label{cor2}
Let $G$ be a countable abelian group and suppose 
$A, B \subset G$ satisfy  
\[
d^*(AB) \leq K \cdot d^*(B) 
\]
for some $K$. Then,
\[
d^*(A^k) \leq K^k \cdot d^*(B), \quad \textrm{for all $k \geq 1$}.
\]
In particular, for every integer $k$, we have
\[
d^*(AB) \geq d^*(A^k)^{\frac{1}{k}} \cdot d^*(B)^{1-\frac{1}{k}}
\]
for all $A, B \subset G$.
\end{corollary}

\begin{corollary}
\label{cor3}
Let $G$ be a countable abelian group and suppose 
$A, B \subset G$ satisfy  
\[
d_*(AB) \leq K \cdot d_*(B) 
\]
for some $K$. Then,
\[
d^*(A^k) \leq K^k \cdot d^*(B), \quad \textrm{for all $k \geq 1$}.
\]
In particular, for every integer $k$, we have
\[
d_*(AB) \geq d^*(A^k)^{\frac{1}{k}} \cdot d_*(B)^{1-\frac{1}{k}}
\]
for all $A, B \subset G$.
\end{corollary}

\subsection{Connection to earlier works}
For a historical survey on the classical Pl\"unnecke estimates for the 
Schnirelmann density on $\bN$, as well as the related estimates for 
the lower asymptotic density, we refer the reader to the lecture notes
\cite{RuSS} by Ruzsa. \\

In this short section we wish to acknowledge that our interest in 
Pl\"unnecke-type estimates for the upper Banach density with 
respect to $\cL_G$ was  spurred by the recent papers \cite{Jin11}
and \cite{Jin11.2} by Jin in which he proves the following special 
case of Corollary \ref{cor2}. 

\begin{theorem}
\label{jins}
For all $A, B \subset \bN$, we have
\[
d^*(AB) \geq d^*(A^k)^{\frac{1}{k}} \cdot d^*(B)^{1-\frac{1}{k}}.
\]
Furthermore, if $d^*(A^k) = 1$, then 
\[
d^*(AB) \geq d^*(B)^{1-\frac{1}{k}}
\qand 
d_*(AB) \geq d_*(B)^{1-\frac{1}{k}},
\]
for all $B \subset \bN$.
\end{theorem}

We should stress that Jin's methods to prove Theorem \ref{jins} are quite different from ours, and it does not seem that they can be extended to give proofs of Corollary \ref{cor2} and Corollary \ref{cor3} for any countable abelian group. 

\subsection{An outline of the proof of Theorem \ref{thm1}}
We shall now attempt to break down the proof of Theorem \ref{thm1} 
into two main propositions which will be proved in Section \ref{Section prop1} and Section \ref{Section prop2} respectively. 

\subsubsection{Magnification ratios in $G$-spaces}

Let $G$ be a countable abelian group and suppose $(X,\mu)$ is a 
(not necessarily ergodic) $G$-space. Given a set $A \subset G$, a
Borel measurable set $B \subset X$ of positive $\mu$-measure and 
$\delta > 0$, we define the \emph{magnification ratio of $B$
with respect to the set $A$} by 
\[
c_\delta(A,B) = \inf 
\Big\{ \frac{\mu(AB')}{\mu(B')} \, : \, B' \subset B 
\qand 
\mu(B') \geq \delta \cdot \mu(B) \Big\}.
\]
We adopt an argument by Petridis in \cite{Petridis12} to the setting 
of $G$-spaces as follows.
\begin{proposition}
\label{prop1}
For every set $A \subset G$ and measurable subset $B \subset X$ of 
positive $\mu$-measure, we have
\[
\sup \Big\{ c_\delta(A',B) \, : \, \textrm{$A' \subset A^{k}$ is finite} \Big\} \leq 
(1-\delta)^{-k} \cdot \Big(\frac{\mu(AB)}{\mu(B)}\Big)^{k}
\]
for every integer $k$ and for all $\delta > 0$.
\end{proposition}

\subsubsection{An ergodic min-max theorem}
For the second step in our proof, we shall assume that $(X,\mu)$ is an
ergodic $G$-space.  Hence, if $A \subset G$ is an ergodic set with 
respect to $(X,\mu)$, then $\mu(AB) = 1$, whenever $B$ has positive 
$\mu$-measure. The following proposition shows that this expansion 
to co-nullity necessarily happens \emph{uniformly} for all Borel measurable sets in $X$ of a given \emph{positive} $\mu$-measure.  

\begin{proposition}
\label{prop2}
Let $G$ be a countable (not necessarily abelian) group and suppose 
$A \subset G$ is an ergodic set. For any $0 < \delta \leq 1$, we 
have
\[
\sup \Big\{ c_\delta(A',B) \, ; \, \textrm{$A' \subset A$ is finite} \Big\}
= 
\frac{1}{\mu(B)}
\]
for every measurable subset $B \subset X$ of positive $\mu$-measure.
\end{proposition}

\subsubsection{Proof of Theorem \ref{thm1}}
Let $G$ be a countable abelian group and let $A \subset G$ be an ergodic basis of order $k$. Suppose $(X,\mu)$ is an ergodic $G$-space and
fix $0 < \delta < 1$. By Proposition \ref{prop1}, the inequality
\[
\sup \Big\{ c_\delta(A',B) \, ; \, \textrm{$A' \subset A^{k}$ is finite} \Big\}
\leq 
(1-\delta)^{-k} \cdot \Big(\frac{\mu(AB)}{\mu(B)}\Big)^{k},
\]
holds, and by Proposition \ref{prop2}, we have
\[
\frac{1}{\mu(B)} 
= 
\sup \Big\{ c_\delta(A',B) \, ; \, \textrm{$A' \subset A^{k}$ is finite} \Big\}.
\]
Combining these two results, we get
\[
\frac{1}{\mu(B)}
\leq 
(1-\delta)^{-k} \cdot \Big(\frac{\mu(AB)}{\mu(B)}\Big)^{k}. 
\]
Since $\delta > 0$ is arbitrary, we can let it tend to zero and conclude
that 
\[
\mu(AB) \geq \mu(B)^{1-\frac{1}{k}},
\]
which finishes the proof. 

\subsection{An outline of the proof of Theorem \ref{thm2}}
The proof of Theorem \ref{thm2} follows a similar route to the 
one we took to prove Theorem \ref{thm1}. We shall need the 
following analogue of the magnification ratios defined earlier. 

Let $G$ be a countable abelian group and suppose $(X,\mu)$ is
a (not necessarily) ergodic $G$-space. Given $A \subset G$ and 
a measurable subset $B \subset X$ with positive $\mu$-measure
we define
\[
c(A,B) 
= 
\inf
\Big\{ 
\frac{\mu(AB')}{\mu(B')} \, : \, \textrm{$B' \subset B$ and $\mu(B') > 0$} 
\Big\}.
\]
The following analogue of Proposition \ref{prop1} is the main technical
result in this paper and immediately implies Theorem \ref{thm1} and 
Theorem \ref{thm2}. We shall prove it in Section \ref{proofProp2.1}.

\begin{proposition}
\label{prop2.1}
For every $A \subset G$ and measurable set $B \subset X$ of positive 
$\mu$-measure, we have
\[
c(A^k,B) \leq \Big( \frac{\mu(AB)}{\mu(B)}\Big)^{k}
\]
for all $k$. 
\end{proposition}

Another ingredient in the proof of Theorem \ref{thm2} is the following
proposition which will be established in the second appendix.

\begin{proposition}
\label{prop2.2}
Suppose $(X,\mu)$ is an ergodic $G$-space. Then, for every 
$A \subset G$ and measurable set $B \subset X$ of positive 
$\mu$-measure, we have
\[
d^*(A) \leq \mu(AB).
\]
In particular, 
\[
d^*(A) \leq c(A,B) \cdot \mu(B).
\]
\end{proposition}

\subsubsection{Proof of Theorem \ref{thm2}}
Let $G$ be a countable abelian group and suppose $(X,\mu)$ is an 
ergodic $G$-space. Fix $A \subset G$ and a measurable set 
$B \subset X$ with positive $\mu$-measure. 
By Proposition \ref{prop2.1} and Proposition \ref{prop2.2} we have
\[
d^*(A^k) \leq c(A^k,B) \cdot \mu(B) \leq \mu(AB)^k \cdot \mu(B)^{1-k}
\]
for all $k$. In particular, if 
\[
\mu(AB) \leq K \cdot \mu(B)
\]
for some $k$, then
\[
d^*(A^k) \leq K^k \cdot \mu(B),
\]
for all $k$, which finishes the proof.

\subsection{An overview of the paper}
The paper is organized as follows. In Section \ref{Section prop1} we 
adapt a recent argument of Petridis in \cite{Petridis12} to magnification
ratios for $G$-spaces with respect to \emph{finite sets}. We then use a
simple increment argument to establish Proposition \ref{prop1}. 

In Section \ref{Section prop2} we outline a general technique to control
magnification ratios for $G$-spaces with respect to an increasing sequence
of \emph{finite} sets in $G$. As a corollary of this technique, we prove 
Proposition \ref{prop2}.

In Section \ref{proofProp2.1} we further refine the technique developed in Section \ref{Section prop2} in order to handle the slightly modified 
magnification ratios discussed in the last subsection. These refinements 
will then give a proof of Proposition \ref{prop2.1}.

The last part of the paper consists of two appendices devoted to the Correspondence Principle mentioned above. 

\subsection{Acknowledgements}
The authors have benefited enormously from discussions with Benjy Weiss during the preparation of this manuscript, and it is a pleasure to thank him for sharing his many insights with us. We are also very grateful for the many inspiring and enlightening discussions we have had with Vitaly Bergelson, Hillel Furstenberg, Eli Glasner, Kostya Medynets,
Fedja Nazarov, Imre Ruzsa and Klaus Schmidt. 

The present paper is an outgrowth of a series of discussions between the authors which took place at the Schr\"odinger Institute in Vienna during July and August of 2010. These discussions continued at Ohio State University, University of Wisconsin, Hebrew University in Jerusalem, Weizmann institute, IHP Paris, KTH Stockholm and ETH Z\"urich. We thank the mathematics departments at these places for their great hospitality. 

The first author acknowledges support from the European Community seventh 
Framework program (FP7/2007-2012) grant agreement 203418 when he was 
a postdoctoral fellow at Hebrew university, and ETH Fellowship FEL-171-03
since January 2011.

\section{Proof of Proposition \ref{prop1}}
Let $G$ be a countable abelian group and suppose $(X,\mu)$ is a 
(not necessarily ergodic) $G$-space. Given $A \subset G$ and a 
Borel measurable set $B \subset X$ with positive $\mu$-measure, 
we define 
\label{Section prop1}
\[
c(A,B) = \inf 
\Big\{ \frac{\mu(AB')}{\mu(B')} \, : \, B' \subset B 
\qand 
\mu(B') > 0 \Big\}.
\]
A recent combinatorial argument of Petridis in \cite{Petridis12} can be 
adapted to the setting of $G$-spaces to give a proof of the following 
proposition.

\begin{proposition}
\label{prop1.2}
For every \emph{finite} set $A \subset G$ and measurable set $B \subset X$ of positive $\mu$-measure, we have
\[
c(A,B) \geq c(A^{k},B)^{\frac{1}{k}}
\]
for all $k \geq 1$. 
\end{proposition}

A drawback with Petridis argument is that it does not automatically yield
any lower bounds on the $\mu$-measures of the subsets $B' \subset B$
which almost realize the infimum $c(A,B)$. This is taken care of by the 
following increment argument. 

\begin{proposition}
\label{prop1.3}
Let $A \subset G$ be a finite set and let $B \subset X$ be a measurable subset 
of positive $\mu$-measure. Fix $0 < \delta < 1$ and a positive integer $k$. Suppose $B' \subset B$ is a measurable subset which satisfies 
\begin{equation}
\label{wanted ineq}
\frac{\mu(A^{k}B')}{\mu(B')} 
\leq 
(1-\delta)^{-k} \cdot 
\Big( \frac{\mu(AB)}{\mu(B)} \Big)^{k}.
\end{equation}
Then, either 
\[
\mu(B') \geq \delta \cdot \mu(B)
\]
or there exists a measurable set $B' \subset B'' \subset B$, which satisfies 
\eqref{wanted ineq}, such that $B'' \setminus B'$ has positive $\mu$-measure. 
\end{proposition}

\begin{proof}
First note that if $B_1, B_2 \subset B$ are measurable sets with 
\emph{$\mu$-null} intersection, which both satisfy inequality \eqref{wanted ineq}, then so does the set $B_1 \cup B_2$. \\

Assume that $B' \subset B$ satisfies \eqref{wanted ineq} and 
$\mu(B') < \delta \cdot \mu(B)$. Set $B_o = B \setminus B'$
and note that 
\[
\mu(B_o) \geq (1-\delta) \cdot \mu(B)
\]
and thus
\[
\eps_o = 1 - (1-\delta) \cdot \frac{\mu(B)}{\mu(B_o)} > 0.
\]
If we define
\[
\eps_1 = \eps_o \cdot (1-\delta)^{-1} \cdot \frac{\mu(AB)}{\mu(B)},
\]
then, by Proposition \ref{prop1.2}, applied to the measurable set $B_o$ of
positive $\mu$-measure, there exists a measurable set $B_o' \subset B_o$ of
positive $\mu$-measure, such that 
\begin{eqnarray*}
\Big( \frac{\mu(A^{k}B_o')}{\mu(B_o')} \Big)^{\frac{1}{k}}
&\leq & 
\eps_1 + \frac{\mu(AB_o)}{\mu(B_o)} \\
&\leq &
\eps_1 + (1-\eps_o) \cdot (1-\delta)^{-1} \cdot \frac{\mu(AB_o)}{\mu(B)} \\
&\leq &
\eps_1 + (1-\eps_o) \cdot (1-\delta)^{-1} \cdot \frac{\mu(AB)}{\mu(B)} \\
&= &
(1-\delta)^{-1} \cdot \frac{\mu(AB)}{\mu(B)}.
\end{eqnarray*}
Hence, $B_o' \subset B$ satisfies inequality \eqref{wanted ineq}, and
since $B_o'$ is disjoint to $B'$, we conclude that the set 
\[
B'' = B' \cup B_o' \subset B
\]
also satisfies \eqref{wanted ineq}, which finishes the proof. 
\end{proof}

The proof of Proposition \ref{prop1} is now an almost immediate consequence of the two propositions above. 

\subsubsection{Proof of Proposition \ref{prop1}}
Let $A \subset G$ be a \emph{finite} set and suppose $B \subset X$ is 
a measurable subset of positive $\mu$-measure. Given $0 < \delta < 1$ and a positive integer $k$, we wish to establish the inequality 
\begin{equation}
\label{wish}
c_\delta(A^{k},B)^{\frac{1}{k}} \leq (1-\delta)^{-1} 
\cdot
\frac{\mu(AB)}{\mu(B)},
\end{equation}
where 
\[
c_\delta(A,B) = \inf 
\Big\{ \frac{\mu(AB')}{\mu(B')} \, : \, B' \subset B 
\qand 
\mu(B') \geq \delta \cdot \mu(B) \Big\}.
\]
By Proposition \ref{prop1.2}, we have
\[
c(A^{k},B)^{\frac{1}{k}} 
\leq 
c(A,B) 
<
(1-\delta)^{-1} \cdot \frac{\mu(AB)}{\mu(B)},
\]
and thus there exists a measurable subset $B' \subset B$ of 
positive $\mu$-measure, such that
\begin{equation}
\label{wish2}
\Big( \frac{\mu(A^{k}B')}{\mu(B')} \Big)^{\frac{1}{k}} 
\leq 
(1-\delta)^{-1} \cdot \frac{\mu(AB)}{\mu(B)}.
\end{equation}
If $\mu(B') \geq \delta \cdot \mu(B)$, then
\[
c_\delta(A^{k},B)^{\frac{1}{k}} 
\leq
\Big( \frac{\mu(A^{k}B')}{\mu(B')} \Big)^{\frac{1}{k}} 
\leq 
(1-\delta)^{-1} \cdot \frac{\mu(AB)}{\mu(B)},
\]
which is what we wanted to prove. \\

A potentially problematic case would be when there is \emph{no} measurable
subset $B' \subset B$  at all with the lower bound $\mu(B') \geq \delta \cdot \mu(B)$ and which satisfies \eqref{wish2}. Assume, for the sake of contradiction, that we are in this situation and define
\[
\cE 
= 
\Big\{ 
B' \subset B \, : \, \textrm{$B'$ satisfies \eqref{wish2}},
\Big\}
\]
where we also insist that the subsets $B' \subset B$ are measurable. \\

Since $\cE$ is closed under unions of increasing sequences of sets, by the Principle of Exhaustion (see Lemma 215A in \cite{Fremlin}), there exists a measurable subset $B_\infty$ in $\cE$ such that whenever $B'$ is an element
in $\cE$ with $B_\infty \subset B'$, then $B' \setminus B_\infty$ is a $\mu$-null set. By assumption we have $\mu(B_\infty) < \delta \cdot \mu(B)$, so Proposition \ref{prop1.3} guarantees that we can find a measurable set $B'$ in $\cE$ such that $B_\infty \subset B' \subset B$ and with the property
that $B' \setminus B_\infty$ is \emph{not} a $\mu$-null set. 
However, this contradicts the maximality of $B_\infty$ described above,
so we conclude that $\mu(B_\infty) \geq \delta \cdot \mu(B)$, and 
thus
\[
c_\delta(A^{k},B)^{\frac{1}{k}} 
\leq
\Big( \frac{\mu(A^{k}B_\infty)}{\mu(B_\infty)} \Big)^{\frac{1}{k}} 
\leq 
(1-\delta)^{-1} \cdot \frac{\mu(AB)}{\mu(B)}.
\]
Note that if $A \subset G$ is \emph{any} set and $A'$ is a \emph{finite} subset of $A^{k}$,
then there exists a finite set $A_o \subset A$ such that 
the inclusion $A' \subset A_o^{k}$ holds. Hence,
\[
c_{\delta}(A',B)^{\frac{1}{k}} 
\leq 
c_{\delta}(A_o^{k},B)^{\frac{1}{k}} 
\leq 
(1-\delta)^{-1} \cdot \frac{\mu(A_o B)}{\mu(B)}
\leq 
(1-\delta)^{-1} \cdot \frac{\mu(A B)}{\mu(B)}
\]
for every finite subset $A' \subset A^k$ and thus,
\[
\sup 
\Big\{ 
c_{\delta}(A',B) \, : \, \textrm{$A' \subset A^{k}$ is finite} 
\Big\} 
\leq
(1-\delta)^{-k} \cdot \Big( \frac{\mu(A B)}{\mu(B)} \Big)^{k},
\]
which finishes the proof.

\subsection{Proof of Proposition \ref{prop1.2}}

\begin{lemma}
\label{petridis}
Let $A \subset G$ be a finite set and fix $\eps > 0$. Then, for any set 
$B' \subset B$ such that 
\[
\mu(AB') \leq (1+\eps) \cdot \mu(B') \cdot  c(A,B) 
\]
and for every \emph{finite} set $F \subset G$, we have
\begin{equation}
\label{petridis ineq}
\mu(FAB') \leq \Big( (1+\eps) \cdot \mu(FB') 
+ 
\eps \cdot |F| \cdot \mu(B') \Big) \cdot c(A,B).
\end{equation}
\end{lemma}

\begin{proof}
Note that inequality \eqref{petridis ineq} trivially holds whenever the set $F$ consists of a single point. Our argument now goes as follows. Fix a finite set $F \subset G$ for which \eqref{petridis ineq} holds
and pick $g \in G \setminus F$. We shall prove that \eqref{petridis ineq} then holds for the set $F' = F \cup \{g\}$. \\

Since $G$ is abelian, we have the inclusion
\[
A(B' \cap g^{-1}FB') \subseteq  AB' \cap g^{-1}FAB',
\]
and thus,
\begin{eqnarray*}
F'AB 
&=&
FAB' \cup \Big( gAB' \setminus FAB'\Big) \\
&=&
FAB' \cup g \Big(AB' \setminus \big(AB' \cap g^{-1}FAB' \big) \Big) \\
&\subseteq &
FAB' \cup g \Big(AB' \setminus A(B' \cap g^{-1}FB') \Big).
\end{eqnarray*}
Since $B' \cap g^{-1}FB' \subset B' \subset B$, we have
\[
\mu\big(A(B' \cap g^{-1}FB')\big) \geq \mu\big(B' \cap g^{-1}FB'\big) \cdot c(A,B),
\]
and thus
\begin{eqnarray*}
\mu\big( F'AB' \big) 
&\leq & 
\mu(FAB') + \mu(AB') - \mu(A(B' \cap g^{-1}FB')) \\
&\leq &
\mu(FAB') + \mu(AB') - \mu(B' \cap g^{-1}FB') \cdot c(A,B) \\
&\leq &
\mu(FAB') + \Big( (1+\eps) \cdot \mu(B') - \mu(B' \cap g^{-1}FB')\Big) \cdot c(A,B).
\end{eqnarray*}
Since \eqref{petridis ineq} is assumed to hold for the set $F$, we conclude
that
\[
\mu\big( F'AB' \big) 
\leq
\Big( 
(1+\eps) \cdot \mu(FB') + \eps \cdot |F| \cdot \mu(B') + (1+\eps) \cdot \mu(B') - \mu(B' \cap g^{-1}FB')
\Big) \cdot c(A,B).
\]
Note that
\[
\mu(FB') + \mu(B') - \mu(B' \cap g^{-1}FB') = \mu(F'B)
\]
and thus
\begin{eqnarray*}
\mu\big( F'AB' \big)  
&\leq & 
\Big( 
\mu(F'B') + \eps \cdot |F'| \cdot \mu(B') + \eps \cdot \mu(FB') 
\Big) \cdot c(A,B) \\
&\leq &
\Big( 
(1 + \eps) \cdot \mu(F'B') + \eps \cdot |F'| \cdot \mu(B')  
\Big) \cdot c(A,B),
\end{eqnarray*}
which finishes the proof.
\end{proof}

\begin{remark}
We stress that the inclusion
\[
A(B' \cap g^{-1}FB') \subseteq  AB' \cap g^{-1}FAB',
\]
for all subsets $F, A \subset G$ and $B' \subset B$ is the only instance 
in the proof where we use the assumption that $G$ is an abelian group.
\end{remark}

The following proposition strictly contains Proposition \ref{prop1}.
\begin{proposition}
\label{petridis2}
Let $A \subset G$ be a finite set and fix an integer $k$. Then there exists a 
non-negative constant $D_k$, which only depends on $k$ and $A$, such that whenever $0 < \eps < 1$ and $B' \subset B$ is a measurable set with
\[
\mu(AB') \leq (1+\eps) \cdot \mu(B') \cdot c(A,B),
\]
then
\begin{equation}
\label{petridis induction}
\frac{\mu(A^{k+1}B')}{\mu(B')}
\leq  
(1+\eps)^{k+1} \cdot c(A,B)^{k+1} 
+ 
\eps \cdot D_k \cdot c(A,B)^{k}.
\end{equation}
In particular, 
\[
c(A^{k},B) \leq c(A,B)^{k}
\]
for every integer $k$.
\end{proposition}

\begin{proof}
Note that inequality \eqref{petridis induction} clearly holds for $k = 0$ and $D_o = 0$. Define $D_{-1} = 0$ and assume that the inequality has been established for all integers up to $k$. We wish to prove that the inequality then also holds for $k+1$. \\

By Lemma \ref{petridis}, applied to the set $F = A^{k}$, we have
\[
\mu(A^{k+1}B') 
\leq 
\Big( 
(1+\eps) \cdot \mu(A^{k}B') + \eps \cdot |A|^{k} \cdot \mu(B')
\Big) 
\cdot 
c(A,B).
\]
By our induction assumption, there exists a non-negative constant 
$D_{k-1}$, which only depends on $k$ and the set $A$, such that
\[
\frac{\mu(A^{k}B')}{\mu(B')}
\leq  
(1+\eps)^{k} \cdot c(A,B)^{k} 
+ 
\eps \cdot D_{k-1} \cdot c(A,B)^{k-1},
\]
and thus, since $c(A,B) \geq 1$, we have
\[
\frac{\mu(A^{k+1}B')}{\mu(B')} 
\leq
(1+\eps)^{k+1} \cdot c(A,B)^{k+1} + 
\eps \cdot (2 \cdot D_{k-1} + |A|^{k}) \cdot c(A,B)^{k},
\]
which establishes inequality \eqref{petridis induction} for $k+1$
with $D_k = 2 \cdot D_{k-1} + |A|^{k}$. The last assertion now 
follows upon letting $\eps$ tend to zero. 
\end{proof}

\section{Proof of Proposition \ref{prop2}}
\label{Section prop2}

Let $G$ be a countable (not necessarily abelian) group and suppose $(X,\mu)$ is an \emph{ergodic} $G$-space. Fix a set $A \subset G$ and a 
Borel measurable set $B \subset X$ with positive $\mu$-measure. 
Given $\delta > 0$, we wish to relate the quantities $c_\delta(A,B)$ and
\[
c'_\delta(A,B) = \sup
\Big\{ 
c_\delta(A',B) \, : \, \textrm{$A' \subset A$ is finite} 
\Big\}.
\]
Clearly we always have $c'_\delta(A,B) \leq c_\delta(A,B)$, but it is not 
immediately clear that equality should hold (not even in the case when $A = G$). \\

Fix an increasing exhausting $(A_k)$ of $A$ by finite sets and choose 
a sequence $(B_k)$ of Borel measurable subsets of $B$ with 
$\mu(B_k) \geq \delta \cdot \mu(B)$ such that
\[
c'_\delta(A,B) = \lim_{k \ra \infty} \frac{\mu(A_k B_k)}{\mu(B_k)}.
\]
The aim of this section is to show that if $A \subset G$ is an ergodic set, 
then 
\[
\lim_{k \ra \infty} \frac{\mu(A_k B_k)}{\mu(B_k)} \geq \frac{1}{\mu(B)},
\]
for every measurable set $B \subset X$ and $0 < \delta < 1$. The following proposition supplies the crucial step in the proof. 

\begin{proposition}
\label{prop3}
Let $(A_k)$ be an increasing sequence of \emph{finite} sets in $G$ with union $A$ 
and let $(B_k)$ be a sequence of measurable subsets of $X$ with a uniform lower
bound on their $\mu$-measures. Then there exist a subsequence $(k_i)$ such that the limit
\[
f(x) = \lim_{m \ra \infty} \frac{1}{m} \sum_{i = 1}^{m} \chi_{B_{k_i}}(x)
\]
exists almost everywhere with respect to $\mu$, and if we define the level sets 
\[
E_t = \Big\{ x \in X \, : \, f(x) \geq t \Big\}
\]
for $t \geq 0$, then
\[
\limsup_{k \ra \infty} \frac{\mu(A_{k} B_{k})}{\mu(B_{k})} 
\geq
\frac{\int_{0}^{1} \mu(AE_t) \, dt}{\int_{0}^{1} \mu(E_t) \, dt}. 
\]
\end{proposition}

\subsection{Proof of Proposition \ref{prop2}}
Let $(X,\mu)$ be an ergodic $G$-space and suppose $A \subset G$
is an ergodic set with respect to $(X,\mu)$. If $B \subset X$ is a measurable subset of positive $\mu$-measure and $\delta > 0$, we 
define
\[
c_{\delta}(A,B) = 
\inf
\Big\{ 
\frac{\mu(AB')}{\mu(B')} 
\, : \, 
B' \subset B
\qand
\mu(B') \geq \delta \cdot \mu(B)
\Big\}.
\]
One readily checks that
\[
c_\delta(A',B) \leq c_\delta(A,B) = \frac{1}{\mu(B)}
\]
for every subset $A' \subset A$, so it suffices to show that whenever
$(A_k)$ is an exhaustion of $A$ by finite sets and $(B_k)$ is a 
sequence of measurable subsets of $B$ with
\[
\mu(B_k) \geq \delta \cdot \mu(B)
\]
for all $k$, then
\[
\limsup_{k \ra \infty} \, \frac{\mu(A_{k} B_{k})}{\mu(B_{k})} 
\geq \frac{1}{\mu(B)}.
\]
By Proposition \ref{prop3}, there exists a subsequence $(k_i)$ such that
the limit
\[
f(x) = \lim_{m \ra \infty} \frac{1}{m} \sum_{i=1}^{m} \chi_{B_{k_i}}(x)
\]
exists almost everywhere with respect to $\mu$ and if we define
\[
E_t = \Big\{ x \in X \, : \, f(x) \geq t \Big\} \subset X
\]
for $t \geq 0$, then
\[
\limsup_{k \ra \infty} \, \frac{\mu(A_k B_k)}{\mu(B_k)}
\geq
\frac{\int_{0}^{1} \mu(AE_t) \, dt}{\int_{0}^{1} \mu(E_t) \, dt}. 
\]
Note that the sets $(E_t)$ are decreasing in $t$, so if we define
\[
r = \sup \big\{ 0 \leq t \leq 1 \, : \, \mu(E_t) > 0 \big\},
\]
then $\mu(E_t) > 0$ for all $0 \leq t < r$, and thus
\[
\mu(A E_t) = 1 \quad \forall \, 0 \leq t < r
\]
and $\mu(AE_t) = 0$ for $t > r$ since $A$ is an ergodic set 
with respect to $(X,\mu)$. Hence,
\[
\frac{\int_{0}^{1} \mu(AE_t) \, dt}{\int_{0}^{1} \mu(E_t) \, dt}
= \frac{1}{\frac{1}{r} \int_{0}^{r} \mu(E_t) \, dt}.
\]
Note that since all the sets $B_k$ are assumed to
be subsets of $B$, and the function $f$ is defined as 
an average of the indicator functions of the sets $(B_k)$,
the set 
\[
E_o = \Big\{ x \in X \, : \, f(x) \geq 0 \Big\} 
 \]
must also be a subset of $B$. Since $E_t \subset E_o$ for 
all $0 \leq t \leq 1$, we have
\[
r \cdot \mu(E_o) \geq \int_{0}^{r} \mu(E_t) \, dt = 
\int_X f(x) \, d\mu(x) 
= 
\lim_{m \ra \infty} \frac{1}{m} \sum_{i=1}^{m} \mu(B_{k_i}) \geq \delta \cdot \mu(B).
\]
and thus
\[
\mu(E_o) \geq \frac{\delta}{r} \cdot \mu(B) \geq \delta \cdot \mu(B),
\] 
since $0 < r \leq 1$. \\

In particular, we have
\[
\frac{\mu(AE_o)}{\mu(E_o)} = \frac{1}{\mu(E_o)} \geq c_\delta(A,B) 
\geq
\frac{1}{\frac{1}{r} \int_{0}^{r} \mu(E_t) \, dt}
\geq
\frac{1}{\mu(E_o)},
\]
by the estimates above, so we can conclude that 
\[
c_\delta(A,B) = \frac{1}{\mu(E_o)} \geq \frac{1}{\mu(B)},
\]
which finishes the proof. 
\subsection{Proof of Proposition \ref{prop3}}

Proposition \ref{prop3} will be an easy consequence of the following 
lemma which will be established in the Subsection \ref{proofLemma}.

\begin{lemma}
\label{help prop3}
Let $(A_k)$ be an increasing sequence of \emph{finite} subsets of $G$ and let
$(B_k)$ be a sequence of measurable subsets of $X$ with a uniform lower
bound on their $\mu$-measures. Then there exists a subsequence $(k_i)$
such that the limit
\[
f(x) = \lim_{m \ra \infty} \, \frac{1}{m} \sum_{i = 1}^{m} \chi_{B_{k_i}}(x) 
\]
exists almost everywhere with respect to $\mu$, and if we define the 
level sets
\[
E_t = \Big\{ x \in X \, : \, f(x) \geq t \Big\}
\]
for all $t \geq 0$, then, for all $\eps > 0$ and for every integer $N_o$, there exists $N \geq N_o$ such that
\[
\sup_{i \geq N} \frac{\mu(A_{k_i} B_{k_i})}{\mu(B_{k_i})} \geq (1 - \eps) \cdot \frac{\int_{0}^{1} \mu(A_{k_N} E_t) \, dt}{\int_{0}^{1} \mu(E_t) \, dt}.
\]
\end{lemma}

\subsubsection{Proof of Proposition \ref{prop3}}
Fix an increasing sequence $(A_k)$ of finite sets in $G$ with union $A$
and let $(B_k)$ be a sequence of measurable subsets with a uniform lower
bound on their $\mu$-measures. \\

Fix a decreasing sequence $(\eps_j)$ of positive numbers converging to zero. Lemma \ref{help prop3} guarantees the existence of a subsequence $(k_i)$ such that the limit
\[
f(x) = \lim_{m \ra \infty} \frac{1}{m} \sum_{i=1}^{m} \chi_{B_{k_i}}(x)
\]
exists $\mu$-almost everywhere, and for every $j$, there exists 
$N_j \geq j$ such that
\[
\sup_{i \geq N_j} \frac{\mu(A_{k_i} B_{k_i})}{\mu(B_{k_i})} 
\geq 
(1 - \eps_j) \cdot \frac{\int_{0}^{1} \mu(A_{k_{N_j}} E_t) \, dt}{\int_{0}^{1} \mu(E_t) \, dt}.
\]
Since $(A_k)$ is increasing with union $A$, the $\sigma$-additivity of $\mu$ now implies that
\[
\lim_{j \ra \infty} \mu(A_{k_{N_j}}E_t) = \mu(AE_t)
\]
for all $t \geq 0$, and thus, by dominated convergence, we have
\[
\limsup_{k} \frac{\mu(A_kB_k)}{\mu(B_k)} 
\geq
\lim_{j} \, 
 \sup_{i \geq N_j} \frac{\mu(A_{k_i} B_{k_i})}{\mu(B_{k_i})} 
\geq 
\frac{\int_{0}^{1} \mu(A E_t) \, dt}{\int_{0}^{1} \mu(E_t) \, dt},
\]
which finishes the proof.

\subsection{Proof of Lemma \ref{help prop3}}
\label{proofLemma}

We first show that we can extract a subsequence from
our sequence $(B_k)$ above such that the Cesaro averages of the
corresponding indicator functions converge almost everywhere.
So far, no assumption on the probability space $(X,\mu)$ has been
made. However, in the proof of the following lemma, it will be 
convenient to assume that the associated Hilbert space $L^2(X,\mu)$
is \emph{separable}. For the rest of this subsection, we shall insist on
this assumption. 

\begin{lemma}
\label{ae convergence}
Let $(X,\mu)$ be a probability measure space such that $L^2(X,\mu)$
is a separable Hilbert space and suppose $(\varphi_n)$ is a sequence of uniformly bounded real-valued measurable functions on $X$. Then there exist 
a bounded measurable function $\varphi$, a subsequence 
$(n_k)$ and a conull subset $X' \subset X$ such that
\[
\varphi(x) = \lim_{N \ra \infty} \, \frac{1}{N} \sum_{k=1}^{N} \varphi_{n_k}(x)
\]
for all $x$ in $X'$. 
\end{lemma}

\begin{proof}
By assumption, 
\[
M = \sup_n \|\varphi_n\|_\infty < \infty,
\]
and thus $(\varphi_n)$ is contained in the centered ball of radius $M$
in $L^2(X,\mu)$. By assumption, $L^2(X,\mu)$ is separable and thus its
unit ball is \emph{sequentially} weakly compact, so we can extract a 
subsequence along which $\varphi_n$ converges to a (bounded) 
measurable function $\varphi$ in the weak topology on $L^2(X,\mu)$. 
If we write $\psi_n = \varphi_n - \varphi$, then $\psi_n$ converges 
weakly to the zero function on $X$, and we wish to show that there 
exists a subsequence $(n_k)$ such that
\[
\lim_{N \ra \infty} \frac{1}{N} \sum_{k=1}^{N} \psi_{n_k}(x) = 0
\]
for $\mu$-almost every $x$ in $X$.  Set $n_1 = 1$, and define 
inductively $n_{k}$, for $k > 1$, by
\[
\int_{X} \psi_{n_i}(x) \, \psi_{n_k}(x) \, d\mu(x) \leq \frac{1}{k}
\]
for all $1 \leq i < k$. One readily checks that
\[
\sum_{N \geq 1} 
\int_X 
\Big| 
\frac{1}{N^{2}} \sum_{k=1}^{N^2} \psi_{n_k}(x) \Big|^2 \, d\mu(x) < \infty,
\]
and thus, by the Borel-Cantelli Lemma, there exists a $\mu$-conull subset
$X' \subset X$, such that
\[
\lim_{N \ra \infty} \frac{1}{N^{2}} \sum_{k=1}^{N^2} \psi_{n_k}(x) = 0, 
\quad
\forall \, x \in X'.
\]
Fix $x \in X'$ and let $(L_j)$ be any increasing sequence. Pick a sequence $(N_j)$ with $N_j^{2} \leq L_j < (N_j+1)^{2}$ for all $j$ and note that 
\[
\Big|
\frac{1}{(N_j+1)^{2}} \sum_{k=1}^{(N_j+1)^2} \psi_{n_k}(x) 
- 
\frac{1}{L_j} \sum_{k=1}^{L_j^2} \psi_{n_k}(x) 
\Big|
\leq
5 \cdot M \cdot \frac{N_j}{(N_j+1)^2},
\]
for all $j$. Since the right hand side and the first term on the left hand side both converge to zero as $j$ tends to infinity we conclude that 
\[
\lim_{j \ra \infty} \frac{1}{L_j} \sum_{k=1}^{L_j} \psi_{n_k}(x) = 0,
\]
for all $x$ in $X'$, which finishes the proof. 
\end{proof}

The second lemma asserts that unions of finite translates of \emph{most} level sets of an almost everywhere convergent sequence of functions behave as one would expect. 

\begin{lemma}
\label{ae convergence2}
Let $(X,\mu)$ be a $G$-space and suppose $(\varphi_n)$ is a sequence of  measurable functions on $X$ which converges to a function $\varphi$ almost everywhere with respect to $\mu$. For every real number $t$
and integer $n$, we define the sets
\[
E_t^{n} 
= 
\Big\{ x \in X \, : \, \varphi_n(x) \geq t \Big\}
\qand
E_t 
= 
\Big\{ x \in X \, : \, \varphi(x) \geq t \Big\}.
\]
Then, for every \emph{finite} set $A \subset G$, we have
\[
\lim_{n \ra \infty} \mu(AE_{t}^{n}) = \mu(AE_{t})
\]
for all but countably many $t$.
\end{lemma}

\begin{proof}
By Egorov's Theorem on almost everywhere convergent sequences of 
measurable functions, there exist, for every $\eps > 0$, a measurable 
subset $X_\eps \subset X$ with $\mu(X_{\eps}) \geq 1-\eps$ and 
an integer $n_\eps$ such that 
\[
E_{t+\eps} \cap X_{\eps} \subset E_t^{n} \cap X_\eps \subset E_{t-\eps} \cap X_{\eps}
\]
for all $n \geq n_\eps$ and for every real number $t$. In particular, 
for every finite set $A \subset G$, we have
\[
\mu(AE_{t+\eps}) - \eps \cdot |A| \leq \mu(AE_t^{n}) \leq \mu(AE_{t-\eps}) + \eps \cdot |A|,
\]
for all $n \geq n_\eps$ and for every real number $t$. Hence,
\[
\mu(AE_{t+\eps}) - \eps \cdot |A|
\leq 
\varliminf_{n \ra \infty} \mu(AE_t^{n}) 
\leq
\varlimsup_{n \ra \infty} \mu(AE_t^{n}) 
\leq 
\mu(AE_{t-\eps}) + \eps \cdot |A|,
\]
for all $\eps > 0$. 

Note that the function $t \mapsto \mu(AE_t)$ is decreasing in the 
variable $t$, and hence its set $D$ of discontinuity points is at 
most countable. If $t$ does not belong to $D$, then the left and right
hand side of the inequalities above tend to $\mu(AE_{t})$ as $\eps$
tends to zero, which finishes the proof.
\end{proof}

Recall the setup from the beginning of the subsection. We have fixed an increasing sequence $(A_k)$ of finite subsets of $A \subset G$ and a sequence $(B_k)$ of Borel subsets of $B \subset X$ with $\mu(B_k) \geq \delta \cdot \mu(B)$ such that
\[
c'_\delta(A,B) = \lim_{k \ra \infty} \frac{\mu(A_kB_k)}{\mu(B_k)}.
\]
Our third lemma shows that for large enough $k$, these ratios can 
be approximated from below by ratios of Cesaro averages of the terms which occur in the nominator and denominator. 

\begin{lemma}
\label{asymptotic}
Let $(\beta_k)$ be a sequence of positive real numbers which converges to a
positive real number. Then, for any $\eps > 0$ and for every integer $N_o$, there exists $N \geq N_o$ such that 
for all $n$ and every positive sequence $(\alpha_k)$, the 
inequality 
\[
\frac{1}{n} \sum_{k=N}^{N+n} \frac{\alpha_{k}}{\beta_k} 
\geq
(1-\eps) \cdot 
\Big(
\frac{\alpha_{N} + \ldots + \alpha_{N+n}}{\beta_{N} + \ldots + \beta_{N+n}}
\Big) 
\]
holds. 
\end{lemma}

\begin{proof}
Let $\beta = \lim_k \beta_k > 0$ and fix $\eps > 0$. Choose 
$\delta > 0$ such that 
\[
\frac{1-\delta}{1+\delta} \geq 1 - \eps.
\]
There exists an integer $N$ with the property that 
\[
(1+\delta) \cdot \beta \geq \beta_k \geq (1 - \delta) \cdot \beta 
\]
for all $k \geq N$, and thus
\[
\beta_{N} + \ldots + \beta_{N+n}
\geq
(1-\delta) \cdot \beta \cdot n
\]
for all $n$. In particular, 
\[
\frac{\beta_{N} + \ldots + \beta_{N+n}}{n \cdot \beta_k}
\geq 
\frac{1-\delta}{1+\delta} \geq 1-\eps,
\]
for all $k \geq N$ and for all $n$. \\

Fix a positive sequence $(\alpha_k)$ and note that
\[
\frac{1}{n} \cdot \frac{\alpha_k}{\beta_k} 
\geq 
(1-\eps) \cdot 
\frac{\alpha_k}{\beta_{N} + \ldots + \beta_{N+n}}  
\]
for all $k \geq N$ and for all $n$. Hence,
\[
\frac{1}{n} \sum_{k=N}^{N+n} \frac{\alpha_k}{\beta_k}
\geq
(1-\eps) \cdot 
\frac{\alpha_{N} + \ldots + \alpha_{N+n}}{\beta_{N} + \ldots + \beta_{N+n}},
\]
which finishes the proof.
\end{proof}
The fourth and final lemma before we embark on the proof of Lemma
\ref{help prop3} is a simple inclusion of level sets.

\begin{lemma}
\label{inclusion}
Let $A \subset G$ and suppose $(B_k)$ is a sequence of subsets of $X$. Let 
$(p_k)$ be a summable sequence of positive real numbers and define the 
sets 
\[
E_t = \Big\{x \in X \, : \, \sum_{k} p_k \cdot \chi_{B_k}(x) \geq t \Big\}
\qand 
F_t = \Big\{x \in X \, : \, \sum_{k} p_k \cdot \chi_{AB_k}(x) \geq t \Big\}
\]
for non-negative $t$. Then $AE_t \subset F_t$.
\end{lemma}

\begin{proof}
Suppose $x \in AE_t$, so that $a^{-1}x \in E_t$ for some $a \in A$. 
Then
\[
\sum_{k} p_k \cdot \chi_{AB_k}(x) 
\geq
\sum_{k} p_k \cdot \chi_{B_k}(a^{-1} x) \geq t,
\]
which shows that $x \in F_t$.
\end{proof}

\subsection{Proof of Lemma \ref{help prop3}}
By Lemma \ref{ae convergence}, there exists a subsequence $(k_i)$ such
that the limit 
\[
f(x) = \lim_{m \ra \infty} \frac{1}{m} \sum_{i=1}^{m} \chi_{B_{k_i}}(x)
\]
exists $\mu$-almost everywhere. Clearly, we can also arrange so that the 
sequence $\mu(B_{k_i})$ converges. \\

Fix $\eps > 0$ and an integer $N_o$. Since the $\mu$-measures of the sets $B_k$ are assumed to have a lower bound, Lemma \ref{asymptotic}, applied to the sequences
\[
\alpha_i = \mu(A_{k_i} B_{k_i}) \qand \beta_i = \mu(B_{k_i}),
\] 
guarantees that there exists $N \geq N_o$, such that for all $n$, we 
have
\begin{eqnarray*}
\sup_{i \geq N} \frac{\mu(A_{k_i}B_{k_i})}{\mu(B_{k_i})}
&\geq &
\frac{1}{n} \sum_{i=N}^{N+n} \frac{\mu(A_{k_i}B_{k_i})}{\mu(B_{k_i})} \\
&\geq &
(1-\eps) \cdot 
\frac{\sum_{i=N}^{N+n} \mu(A_{k_i}B_{k_i})}{\sum_{i=N}^{N+n}\mu(B_{k_i})} \\
&\geq &
(1-\eps) \cdot 
\frac{\sum_{i=N}^{N+n} \mu(A_{k_N} B_{k_i})}{\sum_{i=N}^{N+n}\mu(B_{k_i})},
\end{eqnarray*}
where in the last inequality w used the inclusion
$A_{k_N} \subset A_{k_{i}}$ for all $i \geq N$. \\

Define the functions 
\[
f_{n,N}(x) = \frac{1}{n} \sum_{i=N}^{N+n} \chi_{B_{k_i}}
\qand 
h_{n,N}(x) = \frac{1}{n} \sum_{i=N}^{N+n} \chi_{A_{k_N} B_{k_i}}
\]
and their associated level sets
\[
E_t^{n,N} = \Big\{ x \in X \, : \, f_{n,N}(x) \geq t \Big\}
\qand
F_t^{n,N} = \Big\{ x \in X \, : \, h_{n,N}(x) \geq t \Big\}
\]
for $t \geq 0$. By Lemma \ref{inclusion}, we have
\[
A_{k_N} E_{t}^{n,N} \subset F_t^{n,N}
\]
for all $t \geq 0$, and thus
\[
\frac{1}{n} \sum_{i=N}^{N+n} \mu(A_{k_N}B_{k_i}) 
= 
\int_{X} h_{n,N} \, d\mu 
= 
\int_{0}^{1} \mu(F_t^{n,N}) \, dt \geq \int_{0}^{1} \mu(A_{k_N}E_t^{n,N}) \, dt 
\]
and
\[
\frac{1}{n} \sum_{i=N}^{N+n} \mu(B_{k_i}) 
= 
\int_{X} f_{n,N} \, d\mu 
= 
\int_{0}^{1} \mu(E_t^{n,N}) \, dt, 
\]
for all $n$ and $N$. We conclude that
\begin{equation}
\label{lowbound}
\sup_{i \geq N} \frac{\mu(A_{k_i}B_{k_i})}{\mu(B_{k_i})}
\geq 
(1-\eps) \cdot \frac{\int_{0}^{1} \mu(A_{k_N}E_t^{n,N}) \, dt }{\int_{0}^{1} \mu(E_t^{n,N}) \, dt},
\end{equation}
for all $n \geq N$. \\

Note that
\[
\Big|
\frac{1}{N+n} \sum_{i=1}^{N+n} \chi_{B_{k_i}}(x) 
-
\frac{1}{n} \sum_{i=N}^{N+n} \chi_{B_{k_i}}(x) 
\Big|
\leq 
2 \cdot \frac{N}{N+n},
\]
for all $n$ and $N$, and since the first term converges $\mu$-almost
everywhere to $f(x)$ as $n$ tends to infinity, so does the second term,
and hence
\[
f(x) = \lim_{n \ra \infty} f_{n,N}(x)
\]
for all $N$, whenever $f(x)$ exists. Hence, by Lemma \ref{ae convergence2}, 
we have
\[
\lim_{n \ra \infty} \mu(A_{k_N}E_t^{n,N}) = \mu(A_{k_N}E_t)
\]
and
\[
\lim_{n \ra \infty} \mu(E_t^{n,N}) = \mu(E_t)
\]
for all but countably many $t \geq 0$. By dominated convergence, 
we conclude that
\[
\sup_{i \geq N} \frac{\mu(A_{k_i}B_{k_i})}{\mu(B_{k_i})}
\geq 
(1-\eps) \cdot \frac{\int_{0}^{1} \mu(A_{k_N}E_t) \, dt }{\int_{0}^{1} \mu(E_t) \, dt},
\]
upon letting $n$ tend to infinity on the right hand side in \eqref{lowbound}.

\section{Proof of Proposition \ref{prop2.1}}
\label{proofProp2.1}

Proposition \ref{prop2.1} will be an immediate consequence of the 
following proposition.

\begin{proposition}
\label{helpProp2.1}
For every set $A \subset G$ and measurable set $B \subset X$ of 
positive $\mu$-measure, we have
\[
c(A,B) 
\leq 
\sup
\Big\{ 
c_\delta(A',B) \, : \, \textrm{$A' \subset A$ is finite}
\Big\},
\]
for all $\delta > 0$.
\end{proposition}

\begin{proof}[Proof of Proposition \ref{prop2.1}]
Fix $\delta > 0$ and an integer $k$. For every set $A \subset G$ 
and measurable set $B \subset X$ with positive $\mu$-measure,
Proposition \ref{prop1} asserts that 
\[
\sup
\Big\{ 
c_\delta(A',B) \, : \, \textrm{$A' \subset A^k$ is finite}
\Big\}
\leq
(1-\delta)^{-k} \cdot \Big( \frac{\mu(AB)}{\mu(B)} \Big)^k.
\]
By Proposition \ref{helpProp2.1}, we have
\[
c(A^k,B)
\leq 
\sup
\Big\{ 
c_\delta(A',B) \, : \, \textrm{$A' \subset A^k$ is finite}
\Big\},
\]
and thus
\[
c(A^k,B) \leq (1-\delta)^{-k} \cdot \Big( \frac{\mu(AB)}{\mu(B)} \Big)^k.
\]
Since $\delta > 0$ is arbitrary, we are done.
\end{proof}

\subsection{Proof of Proposition \ref{helpProp2.1}}

It will be useful to reformulate Proposition \ref{prop3} as follows. We
shall show how to establish this version in the next subsection. 

\begin{proposition}
\label{helpProp2.1help}
For every set $A \subset G$, measurable set $B \subset X$ with
positive $\mu$-measure and $0 < \delta < 1$, there 
exists a measurable function $f : B \ra [0,1]$ and a probability 
measure $\eta$ on $[0,1]$ such that, if we define 
\[
E_t = \Big\{ x \in B \, : \, f(x) \geq t \Big\}, \quad t \in [0,1],
\]
then
\[
\eta\left(\big\{ t \in [0,1] \, : \, \mu(E_t) > 0 \big\}\right) = 1
\]
and 
\[
\sup\Big\{ c_\delta(A',B) \, : \, \textrm{$A' \subset A$ is finite} \Big\}
\geq \int_0^1 \frac{\mu(AE_t)}{\mu(E_t)} \, d\eta(t).
\]
\end{proposition}

We omit the proof of the following version of Chebyshev's inequality. 

\begin{lemma}
\label{cheb}
Let $(Z,\eta)$ be a probability measure space and suppose $\varphi$
is a non-negative measurable function on $Z$ which does not vanish almost everywhere with respect $\eta$. Then, for every $\eps > 0$, the 
set
\[
Y_\eps 
= 
\Big\{ 
z \in Z \, : \, \varphi(z) < \int_Z \varphi \, d \eta + \eps
\Big\} \subset Z
\]
has positive $\eta$-measure.
\end{lemma}

\begin{proof}[Proof of Proposition \ref{helpProp2.1}]
Fix a set $A \subset G$, a measurable set $B \subset X$ of positive 
$\mu$-measure and a real number $0 < \delta < 1$. By Proposition
\ref{helpProp2.1help}, there exists a non-negative function $f$ on 
$B$ and a probability measure $\eta$ on $[0,1]$, such that if 
we define
\[
E_t = \Big\{ x \in B \, : \, f(x) \geq t \Big\}, \quad t \in [0,1],
\] 
then
\[
\eta\left(\big\{ t \in [0,1] \, : \, \mu(E_t) > 0 \big\}\right) = 1
\]
and 
\[
\sup\Big\{ c_\delta(A',B) \, : \, \textrm{$A' \subset A$ is finite} \Big\}
\geq \int_0^1 \frac{\mu(AE_t)}{\mu(E_t)} \, d\eta(t).
\]
In particular, the function
\[
\varphi(t) = \frac{\mu(AE_t)}{\mu(E_t)}
\]
is strictly positive almost everywhere with respect to $\eta$. Hence, 
by Lemma \ref{cheb} applied to $Z = [0,1]$ and the probability measure
$\eta$ on $[0,1]$ above, we conclude that for every $\eps > 0$, there
exists $t_o$ in $[0,1]$ such that $E_{t_o}$ is not $\mu$-null and 
\[
\int_0^1 \frac{\mu(AE_t)}{\mu(E_t)} \, d\eta(t) 
\geq 
\frac{\mu(AE_{t_o})}{\mu(E_{t_o})} - \eps.
\]
Since $E_{t_o} \subset B$ has positive $\mu$-measure, we have 
\[
\int_0^1 \frac{\mu(AE_t)}{\mu(E_t)} \, d\eta(t) 
\geq 
c(A,B) - \eps,
\]
for all $\eps > 0$, which finishes the proof. 
\end{proof}

\subsection{Proof of Proposition \ref{helpProp2.1help}}
Fix $A \subset G$, a measurable set $B \subset X$ with 
positive $\mu$-measure and a real number $0 < \delta < 1$. Choose an 
increasing sequence $(A_k)$ of finite sets in $G$ and a 
sequence $(B_k)$ of measurable subsets of $B$ with 
$\mu(B_k) \geq \delta \cdot \mu(B)$ for all $k$, such 
that
\[
\sup\Big\{ c_\delta(A',B) \, : \, \textrm{$A' \subset A$ is finite} \Big\}
= 
\lim_{k \ra \infty} \frac{\mu(A_kB_k)}{\mu(B_k)}.
\]
By Proposition \ref{prop3}, there exist a subsequence $(k_i)$ and a
$\mu$-conull set $X' \subset X$ such that the limit
\[
f(x) = \lim_{i \ra \infty} \frac{1}{m} \sum_{i=1}^m \chi_{B_{k_i}}(x)
\]
exists for all $x$ in $X'$, and if we define
\[
E_t = \Big\{ x \in X \, : \, f(x) \geq t \Big\},
\]
then
\[
\lim_{k \ra \infty} \frac{\mu(A_kB_k)}{\mu(B_k)}
\geq 
\frac{\int_0^1 \mu(AE_t) \, dt}{\int_0^1 \mu(E_t) \, dt}.
\]
Since $B_k \subset B$ for all $k$, we see that $f(x) > 0$ only if $x$
belongs to $B$. We shall abuse notation and denote the restriction 
of $f$ to $B$ by $f$ as well. Define
\[
\rho(t) = \frac{\mu(E_t)}{\int_{0}^1 \mu(E_t) \, dt}, 
\quad 
\textrm{for $0 \leq t \leq 1$}
\]
and denote by $\eta$ the probability measure on $[0,1]$ with density $\rho$ with respect to the Lebesgue measure. Since
\[
\frac{\int_0^1 \mu(AE_t) \, dt}{\int_0^1 \mu(E_t) \, dt}
=
\int_0^1 \frac{\mu(AE_t)}{\mu(E_t)} \, d\eta(t),
\]
we are done. 

\section{Appendix I: A Correspondence Principle for product sets}
The aim of this appendix is to outline a complete proof of the 
Correspondence Principle stated in Subsection \ref{corr}. \\

Let $G$ be a countable group and let $\cM(G)$ denote the set
of means on $G$, which is a weak*-closed and 
convex subset of the dual of $\ell^{\infty}(G)$. Given $\lambda$ 
in $\cM(G)$, we can associate to it a \emph{finitely additive} 
probability measure $\lambda'$ by
\[
\lambda'(C) = \lambda(\chi_C), \quad C \subset G. 
\]
Fix a weak*-compact and convex subset $\cC \subset \cM(G)$.
If $X$ is a compact hausdorff space, equipped with an action of 
$G$ by homeomorphisms of $X$, such that there exists a point 
$x_o$ in $X$ with a dense $G$-orbit, then we have a natural 
unital, injective and left $G$-equivariant C*-algebraic morphism
\[
\Theta_{x_o} : C(X) \ra \ell^{\infty}(G)
\]
given by $\Theta_{x_o}\varphi(g) = \varphi(gx_o)$ for all $g$ in 
$G$ and $\varphi$ in $C(X)$. Hence, its transpose 
$\Theta_{x_o}^*$ maps $\cM(G)$ into $\cP(X)$. 
In particular, if $\mu = \Theta_{x_o}^*\lambda$ and $B \subset X$
is a clopen set (so that the indicator function $\chi_B$ is a continuous
function on $X$), then 
\[
\lambda'(B_{x_o}) = \mu(B),
\]
where
\[
B_{x_o} = \Big\{ g \in G \, : \, gx_o \in B \Big\} \subset G.
\]
Note that every functional of form $\lambda \mapsto \lambda'(C)$
for $C \subset G$ is weak*-continuous, so by the weak*-compactness
of $\cC$, both the supremum and the infimum in the definitions of 
the upper and lower Banach densities with respect to $\cC$ are attained. 
Furthermore, since $\cC$ is also assumed convex, Bauer's Maximum Principle (see e.g. Theorem 7.69 in \cite{AB06}) guarantees that these extremal values are attained at \emph{extremal} elements of $\cC$, i.e.
elements of $\cC$ which cannot be written as non-trivial convex combinations of other elements in $\cC$. Since $\Theta_{x_o}^*$ is 
affine and weak*-continuous on $\cM(G)$, we see that the image 
\[
\cC_{x_o} = \Theta^{*}_{x_o}(\cC) \subset \cP(X)
\]
is weak*-compact and convex. Note that if $A \subset G$ and 
$B \subset X$ is any set, then $AB_{x_o} = (AB)_{x_o}$. Hence, if 
$A' \subset A$ is a finite set and $B$ is a clopen set in $X$, then 
$A'B$ is again clopen in $X$, so that if 
$\mu = \Theta_{x_o}^{*}\lambda$ for some $\lambda$ in $\cC$, then
\[
\lambda'(A'B_{x_o}) = \lambda'((A'B)_{x_o}) = \mu(A'B).
\]
In particular, we have 
\begin{eqnarray*}
d_{\cC}^{*}(AB_{x_o})
&\geq &
\sup 
\Big\{
d^{*}_{\cC}(A'B_{x_o}) \, : \, \textrm{$A' \subset A$ is finite}
\Big\} \\
&\geq &
\sup 
\Big\{
\lambda'(A'B_{x_o}) \, : \, \textrm{$A' \subset A$ is finite}
\Big\} \\
&= &
\sup 
\Big\{
\mu(A'B) \, : \, \textrm{$A' \subset A$ is finite}
\Big\} \\
&= &
\mu(AB),
\end{eqnarray*}
where the last equality holds because $\mu$ is $\sigma$-additive.
We conclude that for every 
$A \subset G$ and clopen set $B \subset X$, there exists an
extremal element $\mu$ in $\cC_{x_o}$ such that 
\[
d^*_{\cC}(B_{x_o}) = \mu(B) \qand d^*_{\cC}(AB) \geq \mu(AB).
\]
The same argument also gives the following inequality. Fix 
$A \subset G$ and a clopen set $B \subset X$. We can find 
an extremal element $\lambda$ in $\cC$ such that 
\[
d^{\cC}_*(AB_{x_o}) = \lambda'(AB_{x_o}).
\]
If we let $\nu = \Theta_{x_o}^*\lambda$, then $\nu$ is an 
extremal element in $\cC_{x_o}$ and
\[
\nu(B) = \lambda'(B_{x_o}) \geq d_*^{\cC}(B_{x_o}).
\]
Furthermore, we have
\begin{eqnarray*}
d_*^{\cC}(AB_{x_o}) 
&=& 
\lambda'(AB_{x_o}) = \lambda'((AB)_{x_o}) \\
&\geq &
\sup
\Big\{ 
\lambda'((A'B)_{x_o}) \, : \, \textrm{$A' \subset A$ is finite} 
\Big\} \\
&=&
\sup
\Big\{ 
\nu((A'B)_{x_o}) \, : \, \textrm{$A' \subset A$ is finite} 
\Big\} \\
&=&
\nu(AB),
\end{eqnarray*}
where the last equality holds because $\nu$ is $\sigma$-additive.
We conclude that whenever $A \subset G$ and $B \subset X$ is a
clopen set, then there exists an extremal element $\nu$ in $\cC_{x_o}$
such that
\[
d_*^{\cC}(B_{x_o}) \leq \nu(B) \qand d_*^{\cC}(B_{x_o}) \geq \nu(AB).
\]
So far, everything we have said works for every compact hausdorff space
$X$, equipped with an action of $G$ by homeomorphisms, every 
clopen subset $B \subset X$ and point $x_o$ with dense orbit. The 
triple $(X,B,x_o)$ gives rise to a set $B_{x_o} \subset G$ and what we 
have seen is that one can estimate product sets of $B_{x_o}$ with any
set $A \subset G$ in terms of the size of the union $AB$ of translates of the set $B$ under the elements in $A$ with respect to certain extremal
elements in $\cC_{x_o}$. We wish to show that every subset $B' \subset G$
is of this form. 

This undertaking is not hard. Let $2^G$ denote the set of all subsets of $G$ equipped with the product topology. Since $G$ is countable, this 
space is metrizable. Note that $G$ acts by homeomorphisms 
on $2^G$ by right translations and the set
\[
U = \Big\{ x \in 2^G \, : \, e \in x \Big\} \subset 2^G
\]
is clopen. Given any set $B' \subset G$, we shall view it as an
element (suggestively denoted by $x_o$) in $2^G$ and we let $X$ 
denote the closure of the $G$-orbit of $x_o$. If we write 
$B = U \cap X$, then $B$ is a clopen set in $X$ and $B_{x_o} = B'$.
We can summarize the entire discussion so far in the following 
\emph{Correspondence Principle}, which essentially dates back to 
Furstenberg \cite{Fu}.

\begin{correspondence}
Given $A, B' \subset G$, there exists a closed $G$-invariant subset $X \subset 2^G$, a clopen set $B \subset X$, a point $x_o$ in $X$ with a dense $G$-orbit and extremal ($\sigma$-additive) probability measures $\mu$ and $\nu$ in $\Theta_{x_o}^{*}(\cC)$ such that 
\[
d^{*}_{\cC}(B') = \mu(B) 
\qand 
d_*^{\cC}(B') \leq \nu(B)
\]
and
\[
d^{*}_{\cC}(AB') \geq \mu(AB) 
\qand 
d^{\cC}_{*}(AB') \geq \nu(AB).
\]
\end{correspondence}

Up until now, the discussion has been very general and no 
assumptions have been made on either the group $G$ or 
the set $\cC$ of means involved. In order for the 
correspondence principle to be useful we need to be able 
to better understand the extremal elements in $\cC$. 

We shall now describe a situation when such an understanding
is indeed possible. Let $G$ be a countable \emph{abelian} group
and denote by $\cL_G$ the set of all \emph{invariant means} on 
$G$. By a classical theorem of Kakutani-Markov, this set is always
non-empty and it is clearly weak*-compact and convex. Given any
compact hausdorff space $X$, equipped with an action of $G$ by 
homeomorphisms of $X$ and containing a point $x_o$ in $X$ with
a dense $G$-orbit, the map 
\[
\Theta_{x_o} : C(X) \ra \ell^{\infty}(G)
\]
defined above is injective and left $G$-equivariant. Hence, its transpose must 
map $\cL_G$ \emph{onto} the space of all $G$-invariant probability
measures on $X$, which we denote by $\cP_G(X)$. It is well-known 
that the extremal elements in $\cP_G(X)$ can be alternatively described 
as the \emph{ergodic} probability measures on $X$, i.e. those $G$-invariant measures which do not admit any $G$-invariant Borel sets 
with $\mu$-measures strictly between zero and one. 

In particular, applying the discussion above to the set $\cC = \cL_G$ and 
adopting the conventions 
\[
d^{*} = d^*_{\cL_G} \qand d_* = d_*^{\cL_G},
\]
we have proved the following version of the Correspondence Principle 
stated in Subsection \ref{corr}.

\begin{corr2}
Let $G$ be a countable abelian group and suppose $A, B' \subset G$. 
Then there exists a closed $G$-invariant subset $X \subset 2^G$, a 
clopen set $B \subset X$ and ergodic $G$-invariant probability measures $\mu$ and $\nu$ on $X$ such that 
\[
d^{*}(B') = \mu(B) 
\qand 
d_*(B') \leq \nu(B)
\]
and
\[
d^{*}(AB') \geq \mu(AB) 
\qand 
d_*(AB') \geq \nu(AB).
\]
\end{corr2}

We stress that this version of the correspondence principle does not apply to the set $\cS \subset \cM(\bZ)$ of Birkhoff means on $\bZ$ as its extremal points are not all mapped to ergodic measures under the map $\Theta_{x_o}^*$ above. 

\section{Appendix II: Proof of Proposition \ref{prop2.2}}

We recall the following simple lemma for completeness. 
\begin{lemma}
\label{nonempty}
Let $(X,\mu)$ be an ergodic $G$-space and suppose $B \subset X$
is a measurable set with positive $\mu$-measure. Then there exists 
a conull set $X' \subset X$ such that the set
\[
B_x = \Big\{ g \in G \, : \, gx \in B \Big\} \subset G
\]
is non-empty for all $x$ in $X'$.
\end{lemma}

\begin{proof}
Since $\mu$ is ergodic and $B \subset X$ has positive $\mu$-measure,
the set $X' = GB$ is $\mu$-conull and $B_x$ is non-empty if and only if
$x$ belongs to $X'$.
\end{proof}

Let $Y$ be a compact $G$-space, i.e. a compact hausdorff space $Y$
equipped with an action of $G$ by homeomorphisms. We say that 
a point $y_o$ is \emph{$G$-transitive} if its $G$-orbit is dense in $Y$.
Recall that if $A \subset Y$ is any subset and $y$ is a point in $Y$, then
we define
\[
A_y = \Big\{ g \in G \, : \, gy \in A \Big\} \subset G. 
\]
Proposition \ref{prop2.2} will follow immediately from the following 
result which is interesting in its own right. \\

\begin{proposition}
\label{helpProp2.2}
Let $(X,\mu)$ be a (not necessarily ergodic) $G$-space. For every clopen set $A \subset Y$, $G$-transitive point $y_o$ in $Y$
and measurable set $B \subset X$, we have
\[
\mu(A_{y_o}^{-1}B) \geq \mu(A_y^{-1}B)
\]
for all $y$ in $Y$.
\end{proposition}

\begin{proof}[Proof of Proposition \ref{prop2.2}]
We shall prove that whenever $(X,\mu)$ is an ergodic $G$-space, 
$Y$ is a compact $G$-space with a $G$-transitive point $y_o$ in 
$Y$ and $A \subset Y$ is a clopen set, then
\[
\mu(A_{y_o}^{-1}B) \geq d^*(A_{y_o})
\]
for every measurable subset $B \subset X$ with positive 
$\mu$-measure. Since every subset $A' \subset G$ can be written on the form $A_{y_o}$
for some compact $G$-space equipped with a $G$-transitive point 
$y_o$ and since
\[
d^*(A_{y_o}) = d^*(A_{y_o}^{-1}),
\]
this will finish the proof of Proposition \ref{prop2.2}. \\

By the Correspondence Principle, we know that we can find an (ergodic) $G$-invariant probability measure $\nu$ on $Y$ such that
\[
d^*(A_{y_o}) = \nu(A).
\]
By Proposition \ref{helpProp2.2}, we have
\begin{eqnarray*}
\mu(A_{y_o}^{-1}B) 
&\geq & 
\int_Y \mu(A_y^{-1}B) \, d\nu(y) \\
&= &
\int_Y \Big( \int_X \chi_{G(A \times B)}(y,x) \, d\mu(x) \Big) \, d\nu(y) \\
&\geq &
\int_X \nu(B_x^{-1}A) \, d\mu(x).
\end{eqnarray*}
Since $(X,\mu)$ is assumed ergodic, Lemma \ref{nonempty} applies and guarantees that there exists a $\mu$-conull $X' \subset X$ such that 
the set $B_x$ is non-empty whenever $x \in X'$. In particular, 
\[
\int_X \nu(B_x^{-1}A) \, d\mu(x) \geq \nu(A) = d^*(A_{y_o}),
\]
which finishes the proof. 
\end{proof}

\subsection{Proof of Proposition \ref{helpProp2.2}}
Note that if $A \subset Y$ is clopen and $F \subset G$ is finite, then
the set
\[
A_F = \Big\{ y \in Y \, : \, F \subset A_y \Big\} = \bigcap_{f \in F} f^{-1}A
\]
is clopen as well. 

Fix $y$ in $Y$, a clopen set $A \subset Y$ and $\eps > 0$. By 
$\sigma$-additivity of $\mu$, there exists a finite set 
$F \subset A_y$ such that 
\[
\mu(A_y^{-1}B) \leq \mu(F^{-1}B) + \eps.
\]
Since $y$ belongs to $A_F$, this set is a non-empty clopen subset of $Y$
and since $y_o$ is $G$-transitive, we conclude that there exists $g$ in $G$
such that $gy_o \in A_F$. In particular, 
\[
F \subset A_{gy_o} = A_{y_o} g^{-1},
\]
which implies
\[
\mu(F^{-1}B) = \mu(g^{-1}F^{-1}B) \leq \mu(A_{y_o}^{-1}B)
\]
and hence
\[
\mu(A_y^{-1}B) \leq \mu(A_{y_o}^{-1}B) + \eps,
\]
for all $\eps > 0$, which finishes the proof. 


\begin{thebibliography}{99}
\bibitem{AB06}
Aliprantis, C. D.; Border, K.C. 
\emph{Infinite dimensional analysis. A hitchhiker's guide.} 
Third edition. Springer, Berlin, 2006. xxii+703 pp. ISBN: 978-3-540-32696-0
\bibitem{Bos}
Boshernitzan, M. D.
\emph{Uniform distribution and Hardy fields.} 
J. Anal. Math. \textbf{62} (1994), 225--240. 
\bibitem{Fremlin}
Fremlin, D. H. 
\emph{Measure theory. Vol. 2. Broad foundations.} 
Corrected second printing of the 2001 original. 
Torres Fremlin, Colchester, 2003. 563+12 pp. (errata). 
ISBN: 0-9538129-2-8
\bibitem{Fu}
Furstenberg, H.
\emph{Ergodic behavior of diagonal measures and a theorem of Szemerédi on arithmetic progressions.} 
J. Analyse Math. \textbf{31} (1977), 204--256. 
\bibitem{Jin11}
Jin, R. \emph{Pl\"unnecke's theorem for asymptotic densities.} 
Trans. Amer. Math. Soc. 363 (2011), no. \textbf{10}, 5059--5070. 
\bibitem{Jin11.2}
Jin, R. \emph{Density versions of Pl\"unnecke's Theorem - Epsilon-Delta Approach}. Preprint.  
\bibitem{Petridis12}
Petridis, G. 
\emph{New proofs of Pl\"unnecke-type estimates for product sets in groups.} Combinatorica 32 (2012), 
no. \textbf{6}, 721--733.
\bibitem{Plunnecke70}
Pl\"unnecke, H. 
\emph{Eine zahlentheoretische Anwendung der Graphentheorie.} (German) J. Reine Angew. Math. \textbf{243} 1970 171--183.
\bibitem{RuSS}
Ruzsa, I. Z. \emph{Sumsets and structure}. Combinatorial number theory and additive group theory. Advanced Courses in Mathematics CRM Barcelona. Basel: Birkh\"auser. pp. 87--210. ISBN 978-3-7643-8961-1.
\end{thebibliography}
\end{document}